\documentclass[12pt]{amsart}

\usepackage{amsmath,amssymb,amsthm,amsfonts}
\usepackage[usenames]{color}
\usepackage{hyperref}

\theoremstyle{plain}
\newtheorem{theorem}{Theorem}
\newtheorem{lemma}{Lemma}

\def\Chi{\mathcal X}

\def\RR{{\mathbb R}}
\def\Vol{\mbox{\rm Vol}}
\def\E{\mathcal{E}}

\begin{document}

\title [Geometric stability of the Coulomb energy]
{Geometric stability \\of the Coulomb energy}

\author{Almut Burchard and Gregory R. Chambers}
\address{A. Burchard, University of Toronto, Department of Mathematics,
\newline 40 St. George Street, Room 6290, Toronto, Canada M5S 2E4 
\newline {\tt almut@$\,$math.toronto.edu}}

\address{G.R. Chambers, University of Chicago, Department of Mathematics,
\newline 5734 S. University Avenue, Room 208 C, Chicago, IL 60637 
\newline {\tt chambers@$\,$math.uchicago.edu} }
\date{July 2, 2015}
\begin{abstract} 
The Coulomb energy of a charge 
that is uniformly distributed on some set
is maximized (among sets of given volume) by balls.
It is shown here that near-maximizers are close to balls.
\end{abstract} 


\maketitle

%~\\[-2cm]

\maketitle

%%%%%%%%%%%%%%%%%%%%%%%%%%%%%%%%%%%%%%%%%%%%%%%%%%%%%%%%%%%%%%%%%%%%
\section{Introduction and main result}

The {\em Coulomb energy} of a charge distribution $f$
on $\RR^3$ is --- up to a multiplicative physical constant ---
given by the singular integral
$$ 
\E(f)=
\int_{\RR^3}\int_{\RR^3} \frac{f(x)f(y)}{|x-y|}\, dxdy\,.
$$
According to the Riesz-Sobolev inequality,
the energy of a positive charge distribution increases 
under symmetric decreasing rearrangement: If $f^*$ is 
radially decreasing and equimeasurable with $f$, then 
\begin{equation}
\label{eq:Riesz}
\E(f)\le \E(f^*)\,.
\end{equation}
The physical reason is that symmetrization
increases the interaction of the charges by reducing
the typical distance between them.  Equality holds 
only if the charge distribution is 
already radially decreasing about some point in 
$\RR^3$~\cite{L-Choquard}.  Is this characterization of
equality cases stable?  If the two sides of 
Eq.~(\ref{eq:Riesz}) almost agree, 
how close must $f$ be to a translate of~$f^*$?

We answer this question for charge distributions that 
are uniform on some set $A\subset \RR^3$ of finite volume. 
Let $A^*$ be the ball of the same volume.
With a slight abuse of notation, denote by
$$
\E(A) = \int_A\int_A |x-y|^{-1}\, dxdy
$$
the Coulomb energy of the uniform charge distribution on $A$.

\begin{theorem} \label{thm:sharp-3}
There exists a constant $c>0$ such that
\begin{equation}
\label{eq:sharp-3} 
\frac{\E(A^*) -\E(A)}{\E(A^*)}
\ge c \left( \inf_\tau \frac{\Vol\bigl((\tau A)\bigtriangleup A^*\bigr)}
{2\Vol(A)}
\right)^2\,.
\end{equation}
for every $A\subset \RR^3$ of finite positive volume.
Here, $\tau$ runs over all translations in $\RR^3$,
and $\bigtriangleup$ denotes the symmetric difference.
\end{theorem}

%The left hand side of Eq.~\eqref{eq:sharp-3}
%measures the deviation of the Coulomb potential
%from its maximal value, and the 
%right hand side measures the asymmetry of the charge distribution.
The exponent 2 is best possible; it is achieved for sets constructed from
the unit ball by removing an annulus
whose outer boundary is the unit sphere,
and adding an annulus of the same volume
whose inner boundary is the unit sphere.

Geometric stability results where a {\em deficit}
(the deviation of a functional from its optimal value)
controls some measure of {\em asymmetry} (the distance from the manifold of 
optimizers) have been established for many classical inequalities.
The first results in that direction, due to Bonnesen in the 1920s,
were quantitative improvements of the isoperimetric
inequality for convex sets in the plane.
Two papers from the early 1990s have inspired much
recent progress.  One is Hall's work on the isoperimetric
inequality in $\RR^n$, where
he proves  stability and raises the question
of optimal exponents~\cite{Hall}; the other is
the result of Bianchi and Egnell
on the stability of the Sobolev inequality
for $||\nabla f||^2$ in dimension $n\ge 3$~\cite{BE}.
We refer the interested reader to the surveys~\cite{O,M}.

Less is known for non-local functionals  that involve convolutions,
even though stability results for those have 
important applications in Mathematical 
Physics~\cite{Carlen}.  In many variational problems 
for integral functionals, one can show by compactness arguments that 
all optimizing sequences must converge
--- modulo the symmetries of the functional ---
to extremals~\cite{BG}, 
but bounds for the asymmetry in terms of the deficit are a different 
matter.
Very recently, Christ has introduced tools
from additive number theory to prove stability of 
the Riesz-Sobolev inequality in one dimension~\cite{Christ}.
Figalli and Jerison have obtained stability results
on the Brunn-Minkowski inequality
for non-convex sets in $\RR^n$~\cite{FJ}.
Fusco, Maggi, and Pratelli have proved stability
of Talenti's inequality for the solutions
of Poisson's equation~\cite[Theorem~2]{FMP-elliptic}.
For the Coulomb energy, Guo conjectured that 
\begin{equation}\label{eq:YG}
\E(f^*)-\E(f) \ge c' \inf_{\tau} \E(f\circ\tau^{-1}\!-\!f^*) 
\end{equation}
with some constant $c'>0$. 
(No normalization is required in this inequality, because
both sides scale in the same way.)
Since the Coulomb kernel is positive definite, the right hand
side can be viewed as the square of a distance.
The relationship between Eqs.~\eqref{eq:sharp-3} and
\eqref{eq:YG} with $f=\Chi_A$ will be
clarified by Lemma~\ref{lem:alpha-YG}. 

The proof of Theorem~\ref{thm:sharp-3}
consists of two parts. After some preliminaries,
we use the reflection positivity of the functional 
and a lemma of Fusco, Maggi, and Pratelli~\cite{FMP}
to reduce the problem to sets that are symmetric
under reflection at the coordinate hyperplanes.
The second part of the proof
requires an estimate for the Newton potential
of symmetric sets. At the end of the paper, we 
briefly discuss stability for other Riesz
kernels and in higher dimensions.

%%%%%%%%%%%%%%%%%%%%%%%%%%%%%%%%%%%%%%%%%%
\section{Notation, and stability in higher dimensions}
%%%%%%%%%%%%%%%%%%%%%%%%%%%%%%%%%%%%%%%%%%

%We start with some definitions.  
By the {\em volume}
of a set $A\subset \RR^n$, denoted $\Vol(A)$,
we mean its $n$-dimensional Lebesgue measure.
The centered open ball of the same volume
is denoted by $A^*$; its radius is
called the {\em volume radius} of $A$,
and denoted by  $R_A$. 
The {\em Fraenkel asymmetry} of $A$ is defined by
\begin{equation}\label{eq:alpha}
\alpha(A) = 
\inf_\tau \frac{\Vol\bigl((\tau A)\bigtriangleup A^*\bigr)}{2\Vol(A)} \,.
\end{equation}
Further, $B_R$ stands for the open ball of radius $R$  centered at 
the origin, and $\omega_n$ for the volume of the unit ball.
The uniform surface measure that is induced on the sphere
$\partial B_r\subset\RR^n$ 
by the ambient Lebesgue measure is denoted by $\sigma$.

We consider functionals of the form
\begin{equation}
\label{eq:E-lambda}
\E(A) = \int_A\int_A |x-y|^{-\lambda}\, dxdy
\end{equation}
with $n\ge 3$ and $\lambda\in [n-2,n)$.
(The classical Coulomb energy corresponds to the
case $n=3$ and $\lambda=1$.)
These functionals share the properties that they
are reflection positive as well as positive definite 
(see~\cite{FL}).
Balls uniquely maximize them among
sets of given volume~\cite{L-Choquard}; balls are also the unique
convex sets for which certain related overdetermined boundary-value
problems have solutions~\cite{Reichel}. By scaling,
\begin{equation}\label{eq:scaling}
\E(A)\le \E(A^*) = \mbox{Constant}\cdot (\Vol(A))^{2-\frac{\lambda}{n}}\,.
\end{equation}
The {\em deficit} of $A$ is defined by
\begin{equation}
\label{def:delta}
\delta(A)=\frac{\E(A^*) -\E(A)}{\E(A^*)}\,.
\end{equation}

Each of the functionals can be expressed
in terms of the corresponding {\em Riesz potential}
\begin{equation}\label{eq:Phi-lambda}
\Phi_A(x) = \int_A |x-y|^{-\lambda}\, dy\,,\qquad x\in\RR^n
\end{equation}
as $\E(A)=\int_A \Phi_A$.  
By the Hardy-Littlewood-Sobolev inequality,
$\Phi_A$ lies in $L^p$ for every $p\ge n/\lambda$.
It is subharmonic on $\RR^n$ and smooth on the complement
of $A$, though discontinuities may occur on $\partial A$.  
The Riesz potential is the unique solution of the pseudodifferential
equation
$$
(-\Delta)^{\frac{n-\lambda}{2}} \Phi =  
\mbox{Constant}\cdot \Chi_A
$$
that decays at infinity. The constant
$c_{n,\lambda}$ can be computed with the help
of the Fourier transform (see~\cite[Theorem 5.9]{LL}).  

The Riesz-Sobolev inequality implies that
\begin{equation}
\label{eq:dom}
\int_E \Phi_A(x)\, dx \le \int_{E^*} \Phi_{A^*}(x)\, dx
\end{equation}
for every set $E\subset\RR^n$ of finite volume
(see~\cite[Theorem 3.6]{LL}).
In particular, $\Phi_{A^*}$ is radially decreasing, and 
\begin{equation} \label{eq:max}
\sup_{x} \Phi_A(x)\le \Phi_{A^*}(0) = \int_{A^*}|y|^{-\lambda}\, dy
= \frac{n\omega_n}{n-\lambda}R_A^{n-\lambda}\,.
\end{equation}

Our proof of Theorem~\ref{thm:sharp-3} fails
in higher dimensions, because the crucial lower
bound in Lemma~\ref{lem:key-3}
becomes negative.  Nevertheless, we expect that
the conclusion should hold
--- with the sharp exponent 2 
and suitable constants $c_{n,\lambda}$ ---
for the entire family of functionals in Eq.~\eqref{eq:E-lambda}
with $n\ge 1$ and positive $\lambda\in [n-2,n)$. 

When $n\ge 3$ and $\lambda=n-2$, we call $\E(A)$ the {\em Coulomb energy} and 
$\Phi_A$ the {\em Newton potential} associated with the
uniform charge distribution on $A$.  The Newton potential 
has many special properties related to Poisson's equation
$$
-\Delta \Phi_A =  n(n-2)\omega_n \Chi_A\,.
$$
It is harmonic on the complement of $A$,
subharmonic on $\RR^n$, and satisfies the Gauss law.  
For later use, we compute the Newton 
potential of the centered ball of radius $R$ as
\begin{equation}
\label{eq:Phi-sym}
\Phi_{B_R}(x) = \omega_n R^2 \cdot \left\{
\begin{array}{ll}
\frac{n}{2} - \frac{n-2}{2}\bigl(\frac{|x|}{R}\bigr)^2\,,
\quad & |x|\le R\,,
\\[0.2cm]
\bigl(\frac{|x|}{R}\bigr)^{-(n-2)}\,,& |x|\ge R\,,
\end{array}\right.
\end{equation}
and its Coulomb energy as
$$ 
\E(B_R) = \frac{2n}{n+2} \omega_n^2 R^{n+2}
= \frac{4}{n+2} \Vol(B_R) \cdot \Phi_{B_R}(0)\,.
$$

A remarkable fact is
{\em Talenti's comparison principle},
which says that the symmetric decreasing rearrangement of 
the Newton potential of a charge distribution
is {\em pointwise} smaller than the potential resulting from
symmetrizing the charge distribution itself~\cite{Talenti},
\begin{equation} 
\label{eq:Talenti}
(\Phi_A)^*(x)\le \Phi_{A^*}(x)\,,\qquad x\in\RR^n\,.
\end{equation}
A similar inequality holds between the gradients
of these functions.
The inequalities are strict, unless $A$ is 
essentially a ball~\cite[Theorem~1]{FMP-elliptic}.

Eq.~\eqref{eq:Talenti}
is clearly stronger than the integrated version in 
Eq.~\eqref{eq:dom}.  We will use Talenti's comparison principle to
prove the following result.

\begin{theorem} \label{thm:main}
Let $\E$ be defined by Eq.~\eqref{eq:E-lambda}
on $\RR^n$ with $\lambda=n-2$.
For each $n\ge 3$, there exists a constant $c_n$ such that
\begin{equation}\label{eq:main}
\frac{\E(A^*)-\E(A)}{\E(A^*)}
\ge c_n \alpha(A) ^{n+2}
\end{equation}
for every $A\subset\RR^n$ of finite positive volume.
\end{theorem}

Note that the conclusion for $n=3$ is weaker than Theorem~\ref{thm:sharp-3}.

%%%%%%%%%%%%%%%%%%%%%%%%%%%%%%%%%%%%%%%%%%
\section{Preliminary estimates}
%%%%%%%%%%%%%%%%%%%%%%%%%%%%%%%%%%%%%%%%%%

Throughout this section, $A\subset \RR^n$ is a set of finite
positive volume, the functional $\E(A)$ is given by Eq.~\eqref{eq:E-lambda} 
with $\lambda\in [0,n)$, and $\Phi_A$ is the corresponding Riesz potential.
We start by sharpening the bound
on the maximum of $\Phi_A$ from Eq.~\eqref{eq:max}.

\begin{lemma} \label{lem:max}
If $A\subset\RR^n$ has finite positive volume, then
$$
\sup_{x\in\RR^n} \Phi_A(x)
\le \Phi_{A^*}(0) \cdot 
\left(1- \frac{\lambda(n-\lambda)}{n^2} \alpha(A)^2 
\right)\,.
$$
\end{lemma}

\begin{proof} By scaling, we may take
$A^*$ to be the unit ball. For $x\in\RR^n$,
$$
\Phi_{A^*}(0)-\Phi_A(x)
= \int_{A^*\setminus (x-A)} |y|^{-\lambda}\, dy - 
\int_{(x-A)\setminus A^*} |y|^{-\lambda}\, dy\,.
$$
If $\alpha(A)=\alpha$, then each of the two regions of integration 
has volume  at least $\omega_n \alpha$.  
The first integral is minimized when
$A^*\setminus (x-A)$ is an annulus whose outer boundary
is the unit sphere, and the second integral
is maximized when $(x-A)\setminus A^*$
is an annulus whose inner boundary is the unit sphere. 
Using annuli of the appropriate volume,
we calculate in polar coordinates
\begin{align*}
\Phi_{A^*}(0)-\Phi_A(x)
&\ge
n\omega_n 
\int_{\left(1-\alpha\right)^{1/n}}^1 \!\!r^{n-1-\lambda}\, dr -
n\omega_n \int_1^{\left(1+\alpha\right)^{1/n}} \!\!\!\! r^{n-1-\lambda}\, dr\\
& = \frac{\lambda(n-\lambda)}{n^2} \Phi_{A^*}(0) 
\int_0^\alpha \int_{-s}^s (1+t)^{-1-\frac{\lambda}{n}}\, dtds\,,
\end{align*}
where we have used the Fundamental Theorem of Calculus 
twice.  By Jensen's inequality, the value of the double integral
exceeds~$\alpha^2$.
\end{proof}

Lemma~\ref{lem:max} is needed for
the proof of Theorem~\ref{thm:main}.  In the next lemma, we use a 
similar estimate to relate $\alpha(A)$ to the notion of asymmetry 
that appears Guo's conjecture, see Eq.~\eqref{eq:YG}.
(It plays no role in the proofs of the main results.)

\begin{lemma} 
\label{lem:alpha-YG}
There exist positive constants $c_{n,\lambda}$ and $C_{n,\lambda}$
such that
$$
c_{n,\lambda} \alpha(A)^4 \le 
\inf_\tau  
\frac{\E(\Chi_A\circ\tau^{-1}\!-\!\Chi_{A*})}{\E(A^*)}
\le C_{n,\lambda}\alpha(A)^{2-\frac{\lambda}{n}}
$$
for every $A\subset\RR^n$ of finite positive volume.
\end{lemma}

\begin{proof} Assume by scaling that $A^*$ is the unit ball,
and set $\alpha=\alpha(A)$. 
For the first inequality, we translate $A$ such that the 
infimum in the middle term is assumed when $\tau$ is the identity.
Since $\E$ extends to a positive definite quadratic 
form on $L^1\cap L^\infty$,
we can use the Cauchy-Schwarz' inequality to obtain
\begin{align*}
\E(\Chi_A\!-\!\Chi_{A^*})^\frac12 \E(A^*)^\frac12
&\ge 
\int\int 
\frac{(\Chi_{A^*}(x)\!-\!\Chi_A(x))\Chi_{A^*}(y)}{ |x-y|^{\lambda}} \,dxdy\\
&=\int_{A^*\setminus A} \Phi_{A^*}(x)\, dx - \int_{A\setminus A^*}
\Phi_{A^*}(x)\, dx\\
&\ge \int_{1-\alpha<|x|^n<1} \Phi_{A^*}(x)\, dx - 
\int_{1<|x|^n<1+\alpha} \Phi_{A^*}(x)\, dx\\
&\ge \mbox{Constant}\cdot\alpha^2\,,
\end{align*}
where the constant depends on $n$ and $\lambda$.
We have used that $\Phi_{A^*}$ is 
strictly radially decreasing 
to replace $A^*\setminus A$ and $A\setminus A^*$ with
annuli. The last line follows since 
the gradient of $\Phi_{A^*}$ vanishes only at $x=0$.

For the second inequality, we translate
$A$ so that the infimum in Eq.~\eqref{eq:alpha}
is assumed at the identity.
The Hardy-Littlewood-Sobolev inequality implies that
$$
\inf_{\tau} \E(\Chi_{A}\circ\tau^{-1}\!-\!\Chi_{A^*})
\le C_{n,\lambda} ||\Chi_{A}-\Chi_{A^*}||_{\frac{2n}{2n-\lambda}}^2
= C_{n,\lambda} \alpha^{2-\frac{\lambda}{n}}\,.
$$
\vskip-1.8\baselineskip
\end{proof}

\smallskip 
We need a few more lemmas for the
proof of Theorem~\ref{thm:sharp-3}.
The following integral representation will
appear several times.

\begin{lemma}\label{lem:outside} 
Let $\rho(r)$ denote the volume radius
of $A\cap B_r$. For any $R>0$,
$$
\E(A^*) - \E(A) 
\ge 2\int_R^\infty \int_{A\cap \partial B_r} 
\Bigl(\Phi_{(A\cap B_r)^*}\Big\vert_{\partial B_{\rho(r)}}
\!\! - \Phi_{A\cap B_r}(x)\Bigr) \, d\sigma(x)\,dr\,.
$$
\end{lemma}

\begin{proof}  The functional can be written as
\begin{align}
\notag \E(A)  &= 2\int_A\int_A \Chi_{\{|x|>|y|\}}\, |x-y|^{-\lambda}\, dydx\\
\label{eq:split}
&= 2\int_{A\cap B_R} \Phi_{A\cap B_{|x|}}(x)\,dx
 + 2 \int_{A\setminus B_R} \Phi_{A\cap B_{|x|}}(x)\, dx\\
\notag &= \E(A\cap B_R)
+ 2 \int_R^\infty \int_{A\cap \partial B_r}
\Phi_{A\cap B_r}(x) \, d\sigma(x)\,dr\,.
\end{align}
Applying Eq.~\eqref{eq:split} 
to $A^*$ with $\rho(R)$ in place of $R$, we see that
\begin{align*}
\E(A^*) 
&= \E(B_{\rho(R)} ) + 2\int_{\rho(R)}^\infty 
\Phi_{B_\rho}\Big\vert_{|x|=\rho}\,
n\omega_n \rho^{n-1}\, d\rho\\
&= \E((A\cap B_R)^*) + 
2\int_R^\infty \Phi_{(A\cap B_r)^*} \Big\vert_{|x|=\rho(r)}
\, \sigma(A\cap \partial B_r)\, dr\,.
\end{align*}
In the first line, we have used that
$B_{\rho(R)}\subset A^*$.
The Jacobian for the change of variables
in the next step %$\rho=\rho(r)$ 
is determined by the relation 
$n\omega_n\rho^{n-1} d\rho = \sigma(A\cap \partial B_r)\,  dr$.
Since $\E(A\cap B_R)\le \E((A\cap B_R)^*)$
by Eq.~\eqref{eq:Riesz}, the claim
follows upon subtracting Eq.~\eqref{eq:split}.
\end{proof}

The next lemma reduces the stability problem to bounded sets.

\begin{lemma} \label{lem:bdd}
For every $n\ge 3$ and $\lambda\in [n-2,n)$
there are positive constants $\alpha_{n,\lambda}$ and
$c_{n,\lambda}$ with the following property.
Given a set $A\subset \RR^n$ of finite positive volume
with $\alpha_0:=\Vol(A\bigtriangleup A^*)/(2\Vol(A))\le\alpha_{n,\lambda}$,
there exists a set $\tilde A$ of the same volume such that
$$
\tilde A\subset \bigl (1+c_{n,\lambda}\alpha_0^{1-\frac\lambda n}\bigr)A^*\,,
\qquad
\frac{\Vol(\tilde A \bigtriangleup A^*)}{2\Vol(\tilde A)} =\alpha_0\,, \qquad 
\delta(\tilde{A})
\le \delta(A)\,.
$$
If $A$ is symmetric about the origin, then so is $\tilde A$.
\end{lemma}

\begin{proof}
By scaling, we may assume that $A^*$ is the unit ball,
i.e., $R_A=1$.  Given $R>(1+\alpha_0)^{1/n}$, determine $r\in (1,R)$ such that
$$
\tilde A = (A\cap B_R) \cup (B_r\setminus A^*)
$$ 
has the same volume as $A$. 
By construction,
$\Vol(\tilde A\bigtriangleup A^*)=\Vol(A\bigtriangleup A^*)$,
and $r\le (1+\alpha_0)^{1/n}$.

We want to choose 
$R$ so that $\E(\tilde A)\ge \E(A)$.
It follows from Eq.~\eqref{eq:split} that
$$
\E(A) \le \E(A\cap B_R) + 2 \Vol(A\setminus B_R)\cdot
\sup_{|x|\ge R}\Phi_A(x)\,.
$$
Since $\Phi_A\le \Phi_{A^*} + \Phi_{A\setminus A^*}$,
Eq.~\eqref{eq:max} implies
$$
\Phi_A(x) \le \Phi_{A^*}(x) + \frac{n\omega_n}{n-\lambda}
\alpha_0^{1-\lambda/n}\,.
$$
Similarly, since $\tilde A\cap A=A\cap B_R$ by construction,
\begin{align*}
\E(\tilde A) 
&= \int_{\tilde A\cap A} \Phi_{\tilde A\cap A}(x)\, dx + 
\int_{\tilde A\setminus A} 2\Phi_{\tilde A\cap A}(x) 
+ \Phi_{\tilde A\setminus A}(x)\, dx\\
&\ge \E(A\cap B_R) + 2 \Vol(A\setminus B_R)\cdot
\inf_{|x|\le r}\Phi_{\tilde A}(x) -\E(\tilde A \setminus A)\,,
\end{align*}
and
$$
\Phi_{\tilde A}(x) \ge \Phi_{A^*}(x) - \frac{n\omega_n}{n-\lambda}
\alpha_0^{1-\lambda/n}\,.
$$
We use Eq.~\eqref{eq:scaling} and the fact that 
$\Vol(\tilde A\setminus A)=\Vol(A\setminus B_R)\le \omega_n\alpha_0$
to estimate
$$
\E(\tilde A\setminus A) \le \mbox{Constant}\cdot\Vol(A\setminus B_R)
\alpha_0^{1-\frac{\lambda}{n}}\,.
$$
Collecting terms, we obtain that
$$
\E(\tilde A)-\E(A)
\ge 2 \Vol(A\setminus B_R)\cdot 
\biggl(\!\Phi_{A^*}\Big\vert^{|x|=(1+\alpha_0)^{1/n}}_{|x|=R}
\!\!\!\!
- \mbox{Constant}\cdot \alpha_0^{1-\frac{\lambda}{n}}
\biggr)\,.
$$
We have used that $\Phi_{A^*}$ is
radially decreasing to replace the inner radius
$r$ by $(1+\alpha_0)^{1/n}$.
Since $\Phi_{A^*}$ is a smooth, strictly
radially decreasing function whose gradient does
not vanish outside $A^*$, there exists a constant
$c_{n,\lambda}$ such that the right hand side
is positive for 
$$R=1+c_{n,\lambda}\alpha_0^{1-\lambda/n}$$
when $\alpha_0$ is sufficiently small. 
\end{proof}

We now introduce reflection symmetries to the problem.
The basic construction is as follows. 
Given a hyperplane that bisects $A$ into two 
parts of equal volume, the set $A$ is replaced by 
the union of one of these parts with its mirror image.
We refer to the two sets that can be obtained in this
way as {\em symmetrizations}
of $A$ at the hyperplane. Clearly,
the symmetrizations have the same volume as~$A$.

\bigskip\noindent {\bf FMP Symmetrization Lemma
\cite[Theorem 2.1]{FMP}.}
{\em 
For every $n\ge 1$ there is a positive
constant $c_n$
with the following property.
Given a set $A\subset \RR^n$ of finite positive volume,
there exists a set $\tilde A$ obtained by successive
symmetrization of $A$ at $n$ orthogonal hyperplanes such that}
$$
\alpha(\tilde A)\ge c_n\alpha(A)\,.
$$

\begin{lemma}
\label{lem:ref-positive}
If $\lambda\in [n-2,n)$, then the set constructed in the FMP lemma
satisfies
$$
\delta(\tilde A) \le 2^n \delta(A)\,.
$$
\end{lemma}

\begin{proof} Consider the two possible symmetrizations
$A_+$ and $A_-$ of
$A$ at a single hyperplane.  Since $\lambda\in [n-2,n)$,
the functional is reflection positive, meaning that
$$
\E(A_+) + \E(A_-) \ge 2 \E(A)\,,
$$
see~\cite[Section 1.1]{FL}.
Using that $(A_+)^*=(A_-)^*=A^*$,  we subtract
both sides of the inequality from $\E(A^*)$ to obtain
$$\delta(A_+) + \delta(A_-) \le 2 \delta(A)\,,
$$
and conclude that the symmetrized sets satisfy
$\delta(A_\pm)\le 2 \delta(A)$.
The claim follows by repeating the construction
$n$ times.
\end{proof}

We translate and rotate $\tilde A$
to a set that is symmetric at the coordinate hyperplanes,
and thus symmetric under $x\mapsto -x$.
Such sets
have the useful property that their asymmetry 
is comparable to their symmetric  difference 
from a  {\em centered} ball~\cite[Lemma 2.2]{FMP}. 
The following estimate for the potential 
is the key to the proof of Theorem~\ref{thm:sharp-3}.

\begin{lemma}  
\label{lem:key-3}
If $A\subset B_r$ is symmetric about the origin, then
$$
\Phi_A(x) \le \Phi_{B_r}(x) - (\sqrt2 r)^{-\lambda}
\Vol(B_r\setminus A)
$$
for all $x\in\partial B_r$.
\end{lemma}

\begin{proof} Let $x\in\partial B_r$ be given. The function
$$
f(y) = \frac{1}{2}\left(
|y-x|^{-\lambda} + |y+x|^{-\lambda}\right)
$$
assumes its minimum at a point on $\partial B_r$ 
equidistant to $x$ and $-x$, and the minimum value is
$(\sqrt2 r)^{-\lambda}$. Since $A$ is symmetric, 
$$
\Phi_A(x) = \int_A 
|x-y|^{-\lambda}\, dy
\ge (\sqrt2 r)^{-\lambda} \Vol(A)\,.
$$
The claim follows by replacing $A$ with $B_r\setminus  A$.
\end{proof}

For the Newton potential of $A\subset B_r$,
Lemma~\ref{lem:key-3} implies that
\begin{align}
\notag
\Phi_{A^*}\Big\vert_{\partial A^*} \!\! - 
\sup\,\Phi_{A}\Big\vert_{\partial B_r}
&\ge
\omega_n \left( R_A^2-r^2 + \frac{r^n-R_A^n}{(\sqrt2 r)^{n-2}}
\right)\\
\label{eq:key-3}
&=\omega_n R_A^2 \bigl(-2+n2^{1-\frac n2}\bigr)
\bigl(\tfrac{r}{R_A}-1\bigr)
+ O\bigl(\tfrac{r}{R_A}-1\bigr)^2
\end{align} 
uniformly in $A$ as $\frac{r}{R_A}\to 1$. 
Note that the leading term is positive in dimension $n=3$.

%%%%%%%%%%%%%%%%%%%%%%%%%%%%%%%%%%%%%%%%%%
\section{Proof of the main results}
%%%%%%%%%%%%%%%%%%%%%%%%%%%%%%%%%%%%%%%%%%

\begin{proof}[Proof of Theorem~\ref{thm:sharp-3}]
We specialize to the case of the
Coulomb energy in $\RR^3$, where $\lambda=1$.
We want to find a constant $c>0$ such that
$\delta(A)\ge c\alpha(A)^2$
for all sets of finite positive volume $A\subset\RR^3$.
By scaling, we may assume that
$\Vol(A)= \omega_3=4\pi/3$, so that
$A^*$ is the unit ball. Since $\alpha(A)\le 1$ by definition, it suffices
to prove the claim for $\alpha$ sufficiently small.

By Lemma~\ref{lem:ref-positive} we may assume that
$A$ is symmetric about the origin.   Therefore,
by \cite[Lemma 2.2]{FMP},
$$
\alpha_0:= \Vol(A\bigtriangleup A^*)/(2\omega_3)
\le 3\alpha(A)\,.
$$
By Lemma~\ref{lem:bdd} we may assume furthermore
that $A$ lies in the ball of radius 
$$
R_0=1+c_{3,1}\alpha_0^{\frac23}\,,
$$
provided that $3\alpha(A)\le \alpha_{3,1}$.  We use Lemma~\ref{lem:outside}
with $R=1$ to see that
\begin{align*}
%\label{eq:integral}
\E(A^*) - \E(A) \ge 2\int_1^{R_0}\int_{A\cap \partial B_r} 
\Phi_{(A\cap B_r)^*}\Big\vert_{\partial B_{\rho(r)}}  \!\!\!\! -
\Phi_{A\cap B_r}(x)\, d\sigma(x)\,dr\,,
\end{align*}
where $\rho(r)$ is the volume radius of $A\cap B_r$.
By Eq.~\eqref{eq:key-3},
the integrand is bounded from
below by
\begin{align*}
\Phi_{(A\cap B_r)^*}\Big\vert_{\partial B_{\rho(r)}} \!\!\!\!  -
\Phi_{A\cap B_r}(x)
&\ge \omega_3\,\inf_{1\le r\le R_0}
\left\{\rho(r)^2-r^2 +\frac{r^3-\rho(r)^3}{\sqrt2 r}\right\}\,.
\end{align*}
The function inside the braces can be written
as a product
$$
%\rho^2-r^2 +\frac{r^3-\rho^3}{\sqrt2 r} = 
(r^3-\rho(r)^3) 
\left(-\frac{r+\rho(r)}{r^2+r\rho(r)+\rho(r)^2}
+ \frac{1}{\sqrt{2}r}\right)\,.
$$
Since the first factor is non-decreasing in $r$,
it is bounded from below by $1-\rho(1)^3=\alpha_0$.
This gives for the integral
$$
\E(A^*)-\E(A)\ge \frac{8\pi}{3}\alpha_0^2 
\cdot \inf_{(1-\alpha_0)^{1/3}\le \rho\le r\le R_0}
\left\{-\frac{r+\rho}{r^2+r\rho+\rho^2}  
+ \frac{1}{\sqrt{2}r} \right\}\,.
$$
The infimum is strictly positive for $\alpha_0$ sufficiently small. 
Since $\alpha_0\ge \alpha(A)$ by definition, the theorem follows.
\end{proof}

The proof of Theorem~\ref{thm:sharp-3} 
used that the Coulomb kernel $|x|^{-1}$ is symmetric
decreasing and reflection positive,
without taking advantage of 
the special properties of the Newton potential.
Since all estimates used in the proof
depend continuously on $\lambda$, the conclusion extends
to nearby values.

\medskip\noindent{\bf Corollary} {\em 
Let $\E_\lambda$ be defined
by Eq.~\eqref{eq:E-lambda} for $n=3$ and $\lambda>1$,
and let $\delta_\lambda$ be the corresponding deficit 
given by Eq.~\eqref{def:delta}.
For every $\lambda$ sufficiently close to 1 there exists a constant
$c_\lambda$ such that
$$
\delta_\lambda(A)\ge c_\lambda\alpha(A)^2
$$
for all $A\subset\RR^3$.}

\medskip 
Finally we turn to the Coulomb energy in dimension $n\ge 3$. 

\begin{proof}[Proof of Theorem~\ref{thm:main}.]
Let $n\ge 3$ and $\lambda=n-2$.
Assume, by scaling, that $A^*$ is the unit ball,
and let $\alpha=\alpha(A)$ be the asymmetry of $A$.
Since $\int_{A} \Phi_A \le \int_{A^*}(\Phi_A)^*$,
$$
\E(A^*)-\E(A)\ge \int_{A^*} \Phi_{A^*} -\bigl(\Phi_A\bigr)^*\,dx\,.
$$
By Talenti's comparison principle, the integrand
is nonnegative. Moreover, by Lemma~\ref{lem:max} and
Eq.~\eqref{eq:Phi-sym},
\begin{align*} \Phi_{A^*}(x) -\bigl(\Phi_A\bigr)^*(x) 
&\ge \Phi_{A^*}(x) - \sup_y \Phi_A(y)\\
&\ge
\frac{n-2}{2} \omega_n
\left(\frac{2\alpha^2}{n} 
-|x|^2 \right)\,.
\end{align*}
Integration yields
\begin{align*}
\E(A^*)-\E(A)
&\ge \frac{n-2}{2}\omega_n \int_{A^*}
\Bigl[\frac{2\alpha^2}{n}-|x|^2\Bigr]_+\, dx\\
& =\frac{n-2}{2n}\E(A^*)\cdot 
\biggl(\frac{\sqrt2 \alpha}{\sqrt n}\biggr)^{n+2} \,,
\end{align*}
which proves Eq.~\eqref{eq:main} with $c_n = (n-2)2^{n/2}/n^{2+n/2}$.
\end{proof}

%\bibliographystyle{plain}
%\bibliography{Coulomb}

\section*{Acknowledgments} The authors thank
Nicola Fusco, Francesco Maggi, and Aldo Pratelli 
for sharing their manuscript~\cite{FMP-elliptic},
and for their hospitality on several occasions.
This work was supported in part by the Federal
Government of Canada through an NSERC CGS Fellowship
(G.R.C.) and a Discovery Grant (A.B.),
and by the Province of Ontario through
an OGS Fellowship (G.R.C.).

\newpage

%%%%%%%%%%%%%%%%%%%%%%%%%%%%%%%%%%%%%%%%%%

\end{document}